\documentclass[11pt]{article}
\usepackage{amssymb,amsmath}
\usepackage{amsthm}
\usepackage{amscd}
\newtheorem{theorem}{Theorem}[section]
\newtheorem{theorem*}{Theorem A\!\!}
\newtheorem{proposition}{Proposition}[section]
\newtheorem{proposition*}{Proposition A\!\!}

\newtheorem{corollary*}{Corollary A\!\!}

\newtheorem{lemma}{Lemma}[section]

\DeclareMathOperator{\Mat}{Mat}
\DeclareMathOperator{\Herm}{Herm}

\DeclareMathOperator{\rank}{rank}

\DeclareMathOperator{\tr}{tr}

\DeclareMathOperator{\res}{res}
\DeclareMathOperator{\sgn}{sgn}

\DeclareMathOperator{\Ln}{Ln}

\begin{document}
\title {Covariant bi-differential operators \\on matrix space}

\author{Jean-Louis Clerc}

\date{January 25, 2016}
\maketitle
\begin{abstract}
A family of bi-differential operators from $C^\infty\big(\Mat(m,\mathbb R)\times\Mat(m,\mathbb R)\big)$ into $C^\infty\big(\Mat(m,\mathbb R)\big)$ which are covariant for the projective action of the group $SL(2m,\mathbb R)$ on $\Mat(m,\mathbb R)$ is constructed, generalizing both the \emph{transvectants} and the \emph{Rankin-Cohen brackets} (case $m=1$).
\bigskip

\centerline {\bf R\'esum\'e}
\smallskip

On construit une famille d'op\'erateurs bi-diff\'erentiels de \break $C^\infty\big(\Mat(m,\mathbb R)\times\Mat(m,\mathbb R)\big)$ dans $C^\infty\big(\Mat(m,\mathbb R)\big)$ qui sont covariants pour l'action projective
du groupe $SL(2m,\mathbb R)$ sur $\Mat(m, \mathbb R)$. Dans le cas $m=1$, cette construction fournit une nouvelle approche  des \emph{transvectants} et des \emph{crochets de Rankin-Cohen}.
\end{abstract}

\footnotemark[0]{2000 Mathematics Subject Classification : 22E45, 43A85}

\section*{Introduction}

Let $X = Gr(m,2m,\mathbb R)$ the Grassmannian of $m$-planes in $\mathbb R^{2m}$, and considerate the projective action of the group $SL(2m,\mathbb R)$ on $X$, given for $g\in G$ and $p\in X$ by $g.p=\{ gv, v\in p\}$. Choose an origin $o$ and let $P$ be the stabilizer of $o$ in $G$. The group $P$ is a maximal parabolic subgroup and $X\sim G/P$. The characters $\chi_{\lambda, \epsilon}$ of $P$ are indexed by $(\lambda, \epsilon)\in \mathbb C\times \{\pm\}$. For $(\lambda, \epsilon) \in \mathbb C\times \{\pm\}$, let $\pi_{\lambda,\epsilon}, $ be the corresponding representation induced from $P$, realized on the space $\mathcal E_{\lambda, \epsilon}$ of smooth sections of the line bundle $ E_{\lambda, \epsilon}  = X\times_{P,\,\chi_{\lambda, \epsilon}} \mathbb C$ (degenerate principal series). For the purpose of this paper, it is more convenient to work with the \emph{noncompact realization} of 
$\pi_{\lambda, \epsilon}$ on a space $\mathcal H_{\lambda, \epsilon}$ of smooth functions  on  $V=\Mat(m,\mathbb R)$.

The \emph{Knapp-Stein intertwining operators} form a meromorphic family (in $\lambda$) of  operators which intertwines  $\pi_{\lambda, \epsilon}$ and $\pi_{2m-\lambda, \epsilon}$ (in our notation). In the non compact picture, for generic $\lambda$, the corresponding operators, denoted by  $J_{\lambda, \epsilon}$ are  convolution operators on $V$ by certain tempered distributions. The properties of this family of operators are presented in section 3 and are mostly consequences of the theory of \emph{local zeta functions} and their functional equation on (the simple real Jordan algebra) $V$. Incidentally, the results for $\epsilon = -1$ seem to be new, at least in the present form.

Let $(\lambda, \epsilon),(\mu,\eta)\in \mathbb C\times \{\pm\}$ and consider the tensor product $\pi_{\lambda,\epsilon} \otimes \pi_{\mu,\eta}$, realized (after completion) on a space $\mathcal H_{(\lambda,\epsilon),(\mu,\eta)}$ of smooth functions on $V\times V$. Because of the \emph{covariance property} (see \eqref{covdet}) of the kernel $k(x,y) = \det(x-y)$ under the diagonal action of $G$ on $V\times V$, the multiplication $M$ by  $\det(x-y)$ intertwines $\pi_{\lambda,\epsilon} \otimes \pi_{\mu,\eta}$ and $\pi_{\lambda-1,-\epsilon} \otimes \pi_{\mu-1,-\eta}$ (Proposition 4.1).

Let $(\lambda, \epsilon), (\mu,\eta) \in \mathbb C\times \{\pm\}$ and consider the following diagram 

\[
\begin{CD}
\mathcal H_{(\lambda,\epsilon), (\mu,\eta)} @> ?>>  \mathcal H_{(\lambda+1,-\epsilon),(\mu+1,-\eta)}\\
@VVJ_{\lambda,\epsilon}\otimes J_{\mu,\eta}V @VVJ_{\lambda+1,-\epsilon}\,\otimes\, J_{\mu+1,-\eta}V\\
\mathcal H_{(2m-\lambda,\epsilon), (2m-\mu,\eta)} @>M> >\mathcal H_{(2m-\lambda-1, -\epsilon),(2m-\mu-1,-\eta)}
\end{CD}
\]
The main result of the paper is a (rather explicit) construction of a \emph{differential operator} on $V\times V$ which completes the diagram (Theorem 4.1).  The proof uses the Fourier transform on $V$ and some delicate calculation specific to the matrix space $V$, based in particular on  \emph{Bernstein-Sato's identities} for $(\det x)^s$ (section 2). Up to some normalization factors, this yields a family of differential operators $D_{\lambda, \mu}$ with polynomial coefficients on $V\times V$, covariant w.r.t. $\big(\pi_{\lambda, \epsilon}\otimes \pi_{\mu,\eta},\pi_{\lambda+1,-\epsilon}\otimes \pi_{\mu+1,-\eta})\big)$. Their expression does not depend on $\epsilon$ and $\eta$, and the family  depends holomorphically on $(\lambda, \mu)$. See also Theorem \ref{covHc} for a formulation of the same result in the compact picture.

From this result, it is then easy to construct families of projectively covariant bi-differential operators from $C^\infty(V\times V)$ into $C^\infty(V)$. For any integer $k$, define 
\[B_{\lambda, \mu\,;\, 2k} = \res\circ
D_{\lambda+k, \mu+k}\circ \dots \circ D_{\lambda+1, \mu+1}\circ D_{\lambda,\mu}
\]
where $\res$ is the restriction map from $V\times V$ to the diagonal $diag(V\times V)\sim V$. Clearly, $B_{\lambda, \mu;k}$ is $G$-covariant w.r.t. $(\pi_{\lambda,\epsilon} \otimes \pi_{\mu,\eta}, \pi_{\lambda+\mu+2k, \epsilon \eta})$. For $k$ fixed, the family depends holomorphically on $\lambda, \mu$ and is generically non trivial.

For $m=1$, there is another classical construction of such such projectively covariant bi-differential operators. The \emph{$\Omega$-process}, a cornerstone  in classical invariant theory leads to the construction of the \emph{transvectants}, which are covariant bi-differential operators for special values of the parameters $\lambda$ and $\mu$ connected to the \emph{finite-dimensional representations} of $G=SL(2,\mathbb R)$. The \emph{Rankin-Cohen brackets}, much used in the theory of modular forms, are other examples of such covariant bi-differential operators, for special values of $(\lambda, \mu)$ connected to the \emph{holomorphic discrete series} of $SL(2,\mathbb R)$. 

In case $m=1$, it has been observed  later (see e.g. \cite{ops}) that the  $\Omega$-process can be extended to general $(\lambda,\mu)$, yielding both the transvectants and the Rankin-Cohen brackets as special cases. As computations are easy when $m=1$, the present construction can be  shown to coincide with the approach through the $\Omega$-process, and the operators $B_{\lambda, \mu;k}$ for special of values of $(\lambda, \mu)$, essentially coincide with the transvectants or the Rankin-Cohen brackets. For another related but different point of view see \cite{kp} (specially section 9) and the expository notes \cite{ku}. The situation where $m\geq 2$ is further commented in section 6.

The striking fact that the operator $D_{\lambda, \mu}$, although obtained by composing non-local operators, is a \emph{differential} operator (hence local) was already observed  in another geometric context, namely for conformal geometry on the sphere $S^d, d\geq 3$ (see \cite{bc}, \cite{c}). It seems reasonable to conjecture that similar results are valid for any  (real or complex) simple Jordan algebra and its conformal group (see \cite{bsz} for analysis on these spaces).

The author wishes to thank T. Kobayashi for helpful conversations related to this paper. 
\bigskip

\centerline {\bf Contents of the paper}
\medskip

{\bf 1.} The degenerate principal series for $Gr(m,2m,\mathbb R)$
\medskip

{\bf 2.} Some functional identities in $\Mat(m,\mathbb C)$ and $\Mat(m,\mathbb R)$
\medskip

{\bf 3.} Knapp-Stein intertwining operators
\medskip

{\bf 4.} Construction of the families $D_{\lambda,\mu}$ and $B_{\lambda, \mu;k}$
\medskip

{\bf 5.} The case $m=1$ and the $\Omega$-process
\medskip

{\bf 6.} The general case and some open problems

\section{The degenerate principal series for $Gr(m,2m,\mathbb R)$}

Let $ X=Gr(m,2m;\mathbb R)$ be Grassmannian of $m$-dimensional vector subspaces of $\mathbb R^{2m}$. The group $G=SL(2m,\mathbb R)$ acts transitively on  $X$.

Let $(\epsilon_1,\epsilon_2,\dots, \epsilon_{2m})$ be the standard basis of $\mathbb R^{2m}$ and let
\[p_0 =\bigoplus_{j=m+1}^{2m} \mathbb R \epsilon_j, \qquad p_\infty=\bigoplus_{j=1}^m \mathbb R \epsilon_j
\ .
\]
The stabilizer of $p_0$ in $ G$ is the parabolic subgroup  $P$ given by
\[ P = \Big\{ \begin{pmatrix}a&0\\c&d \end{pmatrix},\ a,d\in GL(m,\mathbb R),\ \det a \det d = 1\Big\}\ ,
\]
and $ X\sim  G/ P$.

Two subspaces $p$ and $q$ in $ X$ are said to be \emph{transverse} if $p\cap q = \{0\}$, and this relation is denoted by $p\pitchfork q$. Let $\mathcal O= \Big\{ p\in X, p\pitchfork p_\infty\Big\}$. Then $\mathcal O$ is a dense open subset of $X$. Any subspace $p$ transverse to $p_\infty$ can be realized as the graph of some linear map $x : p_0 \longrightarrow p_\infty$, and vice versa. More explicitly, any $p\in \mathcal O$ can be realized as
\[p=p_x = \Big\{\begin{pmatrix}x\xi\\ \xi \end{pmatrix},\  \xi\in \mathbb R^m\Big\}\ ,
\]
where $\xi$ is interpreted as a column vector in $\mathbb R^m$ and $x$ is viewed as an element of $V = \Mat(m, \mathbb R)$.

Let $g\in G$ and $x\in V$. The element $g\in G$ is said to be \emph{defined at $x$} if $g.p_x\in \mathcal O$ and then $g(x)$ is defined by $p_{g(x)} = g.p_x$. More explicitly, if $g=\begin{pmatrix}a&b\\c&d \end{pmatrix}$, then 
\[g.p_x = \Big\{\begin{pmatrix} (ax+b)\,\xi\\ (cx+d)\,\xi\end{pmatrix},\xi \in \mathbb R^m\Big\}, 
\]
so that $g$ is defined at $x$ iff $(cx+d)$ is invertible, and then
\[g(x) = (ax+b)(cx+d)^{-1}\ .
\]
Define  $\alpha : G\times V\longrightarrow \mathbb R$ by
\begin{equation} g=\begin{pmatrix}a&b\\c&d \end{pmatrix},\qquad \alpha(g,x) = \det (cx+d)\ .
\end{equation}
The following elementary calculation is left to the reader.

\begin{lemma}
 Let $g,g'\in G$ and $x\in V$, and assume that $g'$ is defined at x and $g$ is defined at $g'(x)$. Then $gg'$ is defined at $x$ and
\begin{equation}\label{cocyclealpha}
\alpha(gg',x) = \alpha\big(g,g'(x)\big) \alpha(g',x)\ .
\end{equation}
\end{lemma}

The map $x\longmapsto p_x$ is a homeomorphism of $V$ onto $\mathcal O$. The reciprocal of this map $\kappa : \mathcal O \rightarrow V$ is a local chart, thereafter called the \emph{ principal chart}. For any $g\in G$, let $\mathcal O_g = g^{-1} (\mathcal O)$ and $\kappa_g : \mathcal O_g \longrightarrow V$ defined by
$\kappa_g = \kappa\circ g$. Then $\big(\mathcal O_g, \kappa_g\big)_{g\in G}$ is an atlas for $X$. 

Let $g=\begin{pmatrix}a&b\\c&d \end{pmatrix}\in G$. Then
\[V_g := \kappa(\mathcal O_g\cap \mathcal O) =\{ x\in V, \det(cx+d)\neq 0\}\  ,
\]
and the change of coordinates between the charts $\mathcal O$ and $\mathcal O_g$ is given by
\[V_g\ni x \quad \longmapsto g(x) = (ax+b)(cx+d)^{-1}\ .
\]
The group $ P$ admits the Langlands decomposition $P =  L \ltimes  N$, where
\[ L  = \Big\{ \begin{pmatrix} a&0\\0&d\end{pmatrix},  \ \det a \det d =1\Big\},\qquad  N = \Big\{t_v= \begin{pmatrix} \mathbf 1_m&0\\v&\mathbf 1_m\end{pmatrix}, v\in V \Big\}\ .
\]
The group $ L$ acts on $V$ by
\[l= \begin{pmatrix} a&0\\0&d\end{pmatrix},\qquad l(x) = axd^{-1}\ .
\]
Let 
\[\overline N = \left\{ \overline n_y = \begin{pmatrix} \mathbf 1_m&y\\0&\mathbf 1_m\end{pmatrix},\quad
y \in V\right\}\sim V\]
be the opposite unipotent subgroup.
The subgroup $ \overline N$ acts on $ V$ by translations, i.e. $\overline n_y(x) = x+y$ for $y\in V$.

Let $\iota = \begin{pmatrix}0&\mathbf 1_m\\-\mathbf 1_m&0\end{pmatrix}$ be the \emph{inversion}. It is defined on the open set $ V^\times$ of invertible matrices  and acts by
$\iota (x) = -x^{-1}$. Its differential $D\iota(x)$ is given by
$V\ni u \longmapsto  D\iota(x) u = x^{-1}ux^{-1}$.

It is a well-known result that $G$ is generated by $ L,  \overline N$ and $\iota$ (a special case of a theorem valid for the \emph{conformal group} of a simple (real or complex) Jordan algebra).

An element $g= \begin{pmatrix}  a&b\\c&d \end{pmatrix}\in G$ belongs to $\overline N P$ iff  $\det d\neq 0$ and then the following \emph{Bruhat decomposition} holds  
\begin{equation}\label{Bruhat}
\begin{pmatrix}  a&b\\c&d \end{pmatrix}= \begin{pmatrix} \mathbf 1_m&bd^{-1}\\0&\mathbf 1_m\end{pmatrix} \begin{pmatrix} a-bd^{-1} c&0\\c&d\end{pmatrix}\ .
\end{equation}

Let $\chi$ be the character of $P$ defined by
\begin{equation}\label{defchi}
 P\ni p =\begin{pmatrix}a&0\\c&d \end{pmatrix} ,\qquad \chi(p) = \det a =(\det d)^{-1}\ .
\end{equation}

\begin{lemma}
Let $g= \begin{pmatrix}a&b\\c&d\end{pmatrix} \in  G, x\in  V$ and assume that $g$ is defined at $x$. 

$i)$ the differential $Dg(x)$ belongs to $ L$
\smallskip

$ii)$ $\chi(Dg(x)) = \alpha(g,x)^{-1}$

\smallskip

$iii)$ the Jacobian of $g$ at $x$ is equal to
\begin{equation}\label{jacg}
j (g,x) =  \chi\big(Dg(x)\big)^{2m}= \alpha(g,x)^{-2m}\ .
\end{equation}
\end{lemma}
\begin{proof}
By elementary calculation, the statements are verified for elements of $ N, L$ and for $\iota$. As these elements generate $G$, the conclusion follows by using the cocycle relations satisfied by $\alpha(g,x)$ (see \eqref{cocyclealpha}) and by $\chi(Dg(x)$ or $j(g,x)$ as consequences of the chain rule.
\end{proof}

Let  $\lambda\in \mathbb C$ and $\epsilon \in \{\pm\}$. For $t\in \mathbb R^*$ let $t^{\lambda,\epsilon}$ be defined by
\[t\longmapsto \left\{\begin{matrix} \vert t\vert^\lambda \quad &\text { if }\epsilon = +\\ \sgn(t)\vert t\vert^\lambda\quad  &\text{ if } \epsilon = -\end{matrix}\right.
\]
The map $t\longmapsto t^{\lambda, \epsilon}$ is a smooth character of $\mathbb R^*$, and any smooth character is of this form.

Let $\chi^{\lambda, \epsilon}$ be the character of $P$ defined by
\[\chi^{\lambda, \epsilon} (p) = \chi(p)^{\lambda, \epsilon}\ .
\]

Let $E_{\lambda, \epsilon}$ be the line bundle  over $ X =  G/ P$ associated to the character $\chi^{\lambda, \epsilon}$ of $ P$. Let $\mathcal E_{\lambda,\epsilon}$ be the space of smooth sections of $E_{\lambda,\epsilon}$. Then $ G$ acts on $\mathcal E_{\lambda,\epsilon}$ by the natural action
on the sections of $E_{\lambda,\epsilon}$ and gives raise to a representation $\pi_{\lambda, \epsilon}$ of $G$ on $\mathcal E_{\lambda, \epsilon}$. 

A smooth section of $E_{\lambda, \epsilon}$ can be realized as a smooth function $F$ on $G$ which satisfies
\[F(gp) = \chi(p^{-1})^{\lambda, \epsilon} F(g)\ .
\]
To each such function $F$, associate its restriction to $\overline N$, which can be viewed as a function $f$ on $V$ defined for $y\in V$ by 
\[f(y) = F(\overline n_y) = F\left(\begin{pmatrix} \mathbf 1_m&y\\0&\mathbf 1_m\end{pmatrix}\right)\ .
\]
Using the Bruhat decomposition \eqref{Bruhat}, the function $F$ can be recovered from $f$ as 
\[F\left( \begin{pmatrix} a&b\\c&d\end{pmatrix} \right) = (\det d)^{\lambda, \epsilon} f( bd^{-1})\ .
\]
The formula is valid for $g\in \overline N P$ and extends by continuity to all of $G$.

This yields the realization of $\pi_{\lambda, \epsilon}$ in the \emph{noncompact picture}, namely  for $g\in G,$ such that $ g^{-1} = \begin{pmatrix} a&b\\c&d\end{pmatrix}$
\[ \pi_{\lambda, \epsilon}(g) f(y) = \big(\det (cy+d)^{-1}\big)^{\lambda,\epsilon} f\big( (ay+b)(cy+d)^{-1}\big)
\]
\[= \alpha(g^{-1},y)^{-\lambda, \epsilon} f(g^{-1}(y))\ .
\]
In the noncompact picture, the representation $\pi_{\lambda, \epsilon}$ is defined on the image $\mathcal H_{\lambda, \epsilon}$ of  $\mathcal E_{\lambda, \epsilon}$ by the principal chart.  The local expression of an element of $\mathcal H_{\lambda, \epsilon}$ is a function $f\in C^\infty(V)$. For $g\in G$, the function \break $x\mapsto \left(\alpha( g,x)^{-1}\right)^{-\lambda, \epsilon} f\big(g(x)\big)$ is a priori defined on the (dense open) subset $\mathcal O_g$ of $V$. Hence 
a (rather nasty) characterization of the space  is as follows : a smooth function $f$ on $V$ belongs to $\mathcal H_{\lambda, \epsilon}$ if and only if, 
\begin{equation}\label{charH}
\forall g\in G,\quad x\mapsto \left(\alpha( g,x)^{-1}\right)^{-\lambda, \epsilon} f\big(g(x)\big)\text {extends as a } C^\infty \text{ function on } V\ .
\end{equation}
Let $(\lambda,\epsilon), (\mu,\eta) \in \mathbb C\times \{\pm\}$, and let $\pi_{\lambda,\epsilon} \boxtimes \pi_{\mu, \eta}$ be the corresponding product representation of $ G\times  G$. The space  of the representation $\mathcal E_{(\lambda,\epsilon),(\mu,\eta)}$ (after completion) is the space of smooth sections of the fiber bundle $E_{\lambda,\epsilon} \boxtimes E_{\mu, \eta}$ over $X\times X$. For the non-compact realization, observe that $\mathcal O^2= \mathcal O\times \mathcal O$ is an open dense set in $ X\times  X$. For any $g\in G$,  let  $\mathcal O^2_g$ be the image of $\mathcal O^2$ under the \emph{diagonal} action of $g^{-1}$, i.e. $\mathcal O_g^2 = \{ g(x),g(y), x\in  \mathcal O, y\in \mathcal O\}$. Then the family $\big(\mathcal O^2_g, g\in G\big)$ is a covering of $ X\times  X$. Using the corresponding atlas, the local expressions in the principal chart $\kappa\otimes \kappa : \mathcal O^2 \rightarrow V\times V$ of $\mathcal E_{(\lambda, \epsilon),(\mu,\eta)}$ is the space $\mathcal H_{(\lambda,\epsilon), (\mu,\eta)}$ of $C^\infty$ functions $f$ on $V\times V$ such that, for any $g\in G$
\begin{equation}\label{Hsection}
 \alpha(g, x)^{-\lambda,\epsilon} f\big(g(x),g(y)\big)  \alpha(g,y)^{-\mu,\eta}
\text{ extends as a } C^\infty\text{function on } V\times  V\ .
\end{equation}
The group $ G\times G$ acts on $\mathcal H_{(\lambda,\epsilon), (\mu,\eta)}$ by
\begin{equation}\label{boxtimes}
(\pi_\lambda\boxtimes \pi_\mu)(g_1,g_2) f(x,y) =  \alpha(g_1^{-1},x) ^{-\lambda,\epsilon} 
 \alpha(g_2^{-1},y)^{-\mu,\eta}f\big(g_1^{-1}(x), g_2^{-1}(y)\big)
\end{equation}
\begin{lemma} Let $g\in  G, x,y\in V$ such that $g$ is defined at $x$ and at $y$. Then 
\begin{equation}\label{covdet}
\det \big(g(x)-g(y)\big) =  \alpha(g,x)^{-1}\, \det (x-y)\, \alpha(g,y)^{-1}\ .
\end{equation}
\end{lemma}
\begin{proof} If $g\in  \overline N$, $g$ acts by translations on $ V$ and hence \eqref{covdet} is trivial. If $g=\begin{pmatrix}a&0\\0&d\end{pmatrix}$, then $g(x) -g(y) = a(x-y)d^{-1}$, $\alpha(g,x) = \alpha(g,y) = \det a^{-1} \det d$ and \eqref{covdet} is easily verified. When $g=\iota$, then
\[\det(-x^{-1}+y^{-1}) = \det(x^{-1}(x-y)y^{-1})= \det x^{-1} \det(x-y) \det y^{-1}
\]
\[\forall v\in  V,\quad D\iota(x)v = x^{-1}v x^{-1} ,\qquad \alpha(\iota,x) = \det x
\]
and \eqref{covdet} follows easily. The cocycle property \eqref{cocyclealpha} satisfied by  $\alpha$ and the fact that $ G$ is generated by $ \overline N,  L$ and $\iota$ imply \eqref{covdet} in full generality.
\end{proof}

\begin{proposition} The function $k(x,y) =  \det(x-y)$ belongs to $\mathcal H_{(-1,-),(-1,-)}$ and is invariant under the diagonal action of $ G$.
\end{proposition}
\begin{proof} Let $x,y\in V$ and $g\in G$ defined at $x$ and $y$. \eqref{covdet} implies 
\[ \alpha(g,x) k(g(x), g(y)) \alpha(g,y) = k(x,y)
\]
which shows that $k$ belongs to $\mathcal H_{(-1,-),(\-1,-)}$ by the criterion \eqref{Hsection}. Further apply \eqref{boxtimes} for $g_1=g_2=g$ to get the invariance of $k$ under the diagonal action of $G$. 
\end{proof}

\section{Some functional identities in $\Mat(m,\mathbb C)$ and $\Mat(m, \mathbb R)$}

Let $\big(\mathbb E, (\,.\,,\,.\,)\big)$ be a complex finite  dimensional Hilbert space. To any holomorphic polynomial $p$ on $\mathbb E$, associate the holomorphic differential operator $p\left(\frac{\partial}{\partial z}\right)$ defined by
\[
p\Big(\frac{\partial}{\partial z}\Big)\, e^{(z,\xi)} = p\left(\overline{\xi}\right) e^{(z,\xi)}\ .
\]
Let $e_1,e_2,\dots, e_n$ is an orthonormal basis, with corresponding coordinates $z_1,z_2,\dots, z_n$. For    $I=(i_1,i_2,\dots,i_n)$  a $n$-tuple of integers, set 
\[z^I = z_1^{i_1} z_2^{i_2} \dots z_n^{i_n},\qquad \Big(\frac{\partial}{\partial z}\Big)^I = \Big(\frac{\partial}{\partial z_1}\Big)^{i_1} \Big(\frac{\partial}{\partial z_2}\Big)^{i_2} \dots \Big(\frac{\partial}{\partial z_n}\Big)^{i_n}\ .\]
Let $p(z) =\sum_{\vert I\vert \leq N} a_I z^I$ be a holomorphic polynomial on $\mathbb E$. Then
\[p\Big(\frac{\partial}{\partial z}\Big) = \sum_{\vert I\vert \leq N} a_I \Big(\frac{\partial}{\partial z}\Big)^I\ .
\] 
Let $\big(E, \langle\,.\,,\,.\,\rangle\big)$ be a finite dimensional Euclidean vector space. To any polynomial $p$ on  $E$ associate the differential operator $p\big(\frac{\partial}{\partial x}\big)$ such that
\[p\big(\frac{\partial}{\partial x}\big) e^{\langle x,\xi\rangle} =p(\xi) e^{\langle x,\xi\rangle}.
\]
\begin{lemma}\label{realcomplex}
 Let $\big(\mathbb E, (\,.\,,\,.\,)\big)$ be a complex finite dimensional Hilbert space, and let  $\big(E, \langle\,.\,,\,.\,\rangle\big)$ be a
real form of $\mathbb E$ such that
\[\forall x,y\in E,\quad (x,y) = \langle x,y\rangle\ .
\]
Let $p$ be a holomorphic polynomial on $\mathbb E$. Let $\mathcal O$ be an open subset of $\mathbb E$ such that $\omega = \mathcal O\cap E\neq \emptyset$. Let $f$ be a holomorphic function $f$ on $\mathcal O$. Then for $x\in \omega$
\begin{equation}\label{pp}
p\Big( \frac{\partial}{\partial z}\Big)f (x) = p\Big(\frac{\partial}{\partial x}\Big)f_{\vert \omega}(x)\ .
\end{equation}

\end{lemma} 

Now let  $\mathbb E = \Mat(m,\mathbb C)= \mathbb V$ with the inner product
$(z,w) = \tr zw^*$. The restriction of this inner product to the real form $E=\Herm(m, \mathbb C)$ is equal to \[\langle x,y\rangle = \tr xy^* = \tr xy = \tr y^tx^t = \tr \overline{yx}= \overline {\tr xy}= \overline {\tr xy^*}= \overline{\langle x , y\rangle}\] and conditions of Lemma \ref{pp} are satisfied. Denote by $\Omega_m\subset E$ the open cone of positive-definite Hermitian matrices. 

Let $k\in \{ 1,2,\dots m\}$. For $z\in \mathbb V$, let $\Delta_k(z)$ be the principal minor of order $k$ of the matrix $z$. Let $\Delta_k^c(z)$ be the $(m-k)$ anti-principal minor of $z$. Both $\Delta_k$ and $\Delta_k^c$  are holomorphic polynomials on $\mathbb V$.

Let $\mathbb V^\times$ be the set of invertible matrices in $\mathbb V$. Let $z_0\in \mathbb V^\times$.  Choose a local determination of $\ln \det  z$ on a neighborhood of $z_0$, and, for $s\in \mathbb C$ define $(\det z)^s=e^{s \ln \det z}$ accordingly. Any other local determination of $\ln \det z$ is of the form $\ln \det z+2ik\pi$ for som $k\in \mathbb Z$, and the associated local determination of $(\det z)^s$ is given by $e^{2ik\pi s} (\det z)^s$.

Recall the \emph{Pochhammer's symbol}, for $s\in \mathbb C, n\in \mathbb N$
\[(s)_0 = 1,\quad  (s)_1 = s,\dots,\quad (s)_n = s(s+1)\dots(s+n-1),\dots 
\]

\begin{proposition}
 For any $z\in \mathbb V^\times$ and for any local determination of $\ln \det $ in a neighborhood of $z$
\begin{equation}\label{BSprincipalminor}
\Delta_k\Big(\frac{\partial}{\partial z}\Big) (\det z)^s = (s)_k \,\Delta_k^c(z) \, (\det z)^{s-1}
\end{equation}
\end{proposition}
\begin{proof}
Let $z_0\in \mathbb V^\times$. Choose an open neighborhood $\mathcal V$ of $z$ contained in $\mathbb V^\times$ which is simply connected and such that $\mathcal V\cap \Omega_m\neq \emptyset$. On $ \Omega_m$, $\det x>0$ so that $\Ln \det z$ (where $\Ln$ is the principal determination of the logarithm  on $\mathbb C\setminus (-\infty,0[$)  is an appropriate determination of $\ln \det z$ in a neighborhood of $\Omega_m$, which can be analytically continued to  $\mathcal V$ and used for defining $(\det z)^s$ on $\mathcal V$. For $x\in \Omega_m$, the identity
\[\Delta_k\big(\frac{\partial }{\partial x}\big) (\det x)^s = (s)_k\, \Delta_k^c(x)\, (\det x)^{s-1}
\]
holds. It is a special case of \cite {fk} Proposition VII.1.6 for the simple Euclidean Jordan algebra $Herm(m,\mathbb C)$. 
By Lemma \ref{realcomplex}, \eqref{BSprincipalminor} is satisfied for  $z\in \mathcal V\cap Herm(m,\mathbb C)$. As both sides of \eqref{BSprincipalminor} are holomorphic functions, \eqref{BSprincipalminor} yields everywhere on $\mathcal V$. But if \eqref{BSprincipalminor} is valid for \emph{some} local determination of $\ln \det z$ it is valid for \emph{any} local determination.
\end{proof}
There is a real version of these identities.
\begin{proposition} The following identity holds for $x\in V^\times $
\begin{equation}\label{realBS}
\Delta_k\left(\frac{\partial}{\partial x}\right) (\det x)^{s,\epsilon} = (s)_k \,\Delta_k^c(x)\, (\det x)^{s-1,-\epsilon}\ .
\end{equation}

\begin{proof} Let $x\in V^\times$ and assume first that $\det x>0$. In a neighbourhood of $x$ in $\mathbb V^\times$ choose $\Ln (\det z)$ as a local determination of $\ln (\det z)$. Then 
$(\det x)^s = \vert \det x\vert^s$ and hence, using Lemma \ref{realcomplex} and \eqref{BSprincipalminor} 
\[\Delta_k\left(\frac{\partial}{\partial x}\right)\bigl \lvert  \det x \bigr \rvert^s = 
(s)_k\, \Delta_k^c(x) \,\bigl \lvert  \det x \bigr \rvert^{s-1}\ .\]

Next assume that $\det x <0$. In a neighborhood of $x$ in $\mathbb V^\times$ choose $\Ln (-\det z) +i\pi$ as a local determination of $\ln (\det z)$.
Then $(\det x)^s = e^{is\pi} \vert \det x \vert^s$, so that, using again Lemma \ref{realcomplex} and \eqref{BSprincipalminor} 
\[e^{is\pi}\Delta_k\left(\frac{\partial}{\partial x}\right)\bigl \lvert  \det x \bigr \rvert^s=e^{i(s-1)\pi} (s)_k \Delta_k^c(x) \vert \det x \vert^{s-1}\ .
\]
The identity \eqref{realBS} follows.
\end{proof}

\end{proposition}

Let $a=(a_{ij})$ be a $m\times m$ matrix with real or complex entries $a_{ij}$. Let $I$ and $J$ be two subsets of $\{1,2,\dots,m\}$ both of cardinality $k, 0\leq k\leq m$. After deleting the $m-k$ rows (resp. the $m-k$ columns) corresponding to the indices not in $I$ (resp. not in $J$), the determinant of the $k\times k$ remaining matrix is the \emph{minor} associated to $(I,J)$ and will be denoted by $\Delta_{I,J}(a)$. For $k=0$, i.e. $I=J=\emptyset$, by convention  $\Delta_{\emptyset,\emptyset}(a) =1$. For $k=m$, $I=J=\{1,2,\dots,m\}$, $\Delta_{I,J}(a) = \det a$.
 
 For $I= \{ i_1<i_2<\dots< i_k\}$, let $\vert I\vert = i_1+i_2+\dots i_k$.
Also denote by $I^c$  the complement  of $I$ in $\{1,2,\dots,m\}$, which is a subset of cardinality $m-k$. Recall the following elementary result.

\begin{lemma}\label{signsigmaI1} Let $I=\{ i_1<i_2<\dots<i_k\}$ be  a subset of $\{1,2,\dots, m\}$ of cardinality $k$. Let $I^c = \{ i'_1<i'_2<\dots<i'_{m-k}\}$. The permutation $\sigma_I$
defined by
\[\sigma_I(1) = i_1, \dots, \sigma_I(k) = i_k,\quad  \sigma_I(k+1) = i'_1,\dots, \sigma_I(m) = i'_{m-k}\ 
\]
has signature equal to $\epsilon(\sigma_I) =(-1)^{\vert I\vert}$.
\end{lemma}
 
 The next lemma is a variation on (and a consequence of)  the previous lemma.
 \begin{lemma}\label{signsigmaI}
 Let $I=\{ i_1<i_2<\dots<i_k\},\quad J=\{ j_1<j_2<\dots<j_k\}$
 be two subsets of $\{1,2,\dots,m\}$ both of cardinality $k$ . Let \[I^c = \{ i'_1<i'_2<\dots<i'_{m-k}\}, \quad J^c =\{ j'_1<j'_2<\dots <j'_{m-k}\}\ .\]
 The permutation $\sigma = \sigma_{I,J}$ given by
 \[\sigma(i_1) = j_1,\dots ,\sigma(i_k)=j_k,\quad  \sigma(i'_1)=j'_1,\dots,\sigma(i'_{m-k}) = j'_{m-k} 
 \]
 has signature $\epsilon(I,J):=\epsilon(\sigma_{I,J}) = (-1)^{\vert I\vert +\vert J\vert}$.
 \end{lemma}
 
A permutation $\sigma$ such that $\sigma(I) = J$ can be written in a unique way as $\sigma =(\tau\vee \tau_c)\circ \sigma_{I,J}$, where $\tau$ is a permutation of $J$ and $\tau_c$ is a permutation of $J^c$, and $\tau\vee \tau_c$ is the permutation of $\{1,2,\dots,m\}$ which coincides on $J$ with $\tau$ and on $J^c$ with $\tau_c$.

\begin{proposition}
 Let $I,J\subset\{1,2,\dots,n\}$ of equal cardinality $k$.
Then, for $x\in \mathbb V^\times$
\begin{equation}\label{Deltaminor}
\partial(\Delta_{I,J})\big(\Delta^s\big)(x)= \epsilon(I,J)(s)_k\, \Delta_{I^c,J^c}(x)\,\Delta(x)^{s-1}\ .
\end{equation} 
\end{proposition}

\begin{proof} By permuting raws and columns properly, the minor $\Delta_{I,J}$ becomes the $k$-th principal minor  and $\Delta^{I^c, J^c}$ becomes the $m-k$ anti-principal minor, up to a sign. Hence \eqref{Deltaminor} is a consequence of  \eqref{BSprincipalminor} and  Lemma \ref{signsigmaI1}.

\end{proof}
\begin{proposition} Let $f,g$ be two smooth functions defined on $\mathbb V$. Then
\begin{equation}\label{Deltaprod}
\det \left( \frac{\partial}{\partial x}\right) (fg) = \sum_{\begin{matrix} I,J\subset\{1,2,\dots,m\}\\ \# I=\#J\end{matrix}} \epsilon(I,J) \Delta_{I,J}\big(\frac{\partial}{\partial x}\big) f\,\,\Delta_{I^c,J^c}\big(\frac{\partial}{\partial x}\big) g
\end{equation}
\end{proposition}
\begin{proof} For $\sigma\in\mathfrak S_m$ 
\[\frac{\partial^m}{\partial a_{1\sigma(1)} \partial a_{2 \sigma(2)}\dots \partial a_{m \sigma(m)}}(fg) = 
\sum_{\begin{matrix} I\subset\{1,2,\dots,m\}\end{matrix}}
\left(\prod_{i\in I} \frac {\partial }{\partial a_{i\sigma(i)}}\right)f \left(\prod_{i\in I^c}\frac {\partial }{\partial a_{i\sigma(i)}}\right) g\ .
\]
Now, given $I\subset \{ 1,2,\dots, m\}$,
\[\sum_{\sigma \in \mathfrak S_m} = \sum_{\begin{matrix} J\subset \{ 1,2,\dots, m\}\\ \# J = \# I\end{matrix}}\sum_{\begin{matrix}\sigma \in \mathfrak S_m\\ \sigma(I)= J\end{matrix}}
\]
so that

\[\partial(\Delta)(fg) = \sum_{\sigma\in \mathfrak S_m} \epsilon(\sigma) \sum_{\begin{matrix} I\subset\{1,2,\dots,m\}\end{matrix}}
\left(\prod_{i\in I} \frac {\partial }{\partial a_{i\sigma(i)}}\right)f \left(\prod_{i\in I^c}\frac {\partial }{\partial a_{i\sigma(i)}}\right) g\]
\[= \sum_{I\subset \{ 1,2,\dots,m\}}\sum_{\begin{matrix}J\subset \{ 1,2,\dots, m\} \\\# I = \# J\end{matrix}}\sum_{\begin{matrix}\sigma\in \mathfrak S_m\\ \sigma(I)=J\end{matrix}} \epsilon(\sigma)
\left(\prod_{i\in I} \frac {\partial }{\partial a_{i\sigma(i)}}\right)f \left(\prod_{i\in I^c}\frac {\partial }{\partial a_{i\sigma(i)}}\right) g\ .
\] 
Let 
\[I=\{ i_1<i_2<\dots<i_k\},\quad J=\{ j_1<j_2<\dots<j_k\}\]
 \[
I^c = \{ i_1'<i_2',\dots< i'_{m-k}\},\qquad J^c = \{ j_1'<j_2',\dots< j'_{m-k}\}\ .\]
As noted after the proof of Lemma \ref{signsigmaI}, a permutation $\sigma$ such that $\sigma(I)=J$ can be written in a unique way as 
\[\sigma = (\tau\vee \tau_c)\circ \sigma_{I,J}
\]
where $\tau\in \mathfrak S(J), \tau_c\in \mathfrak S(J^c)$.
Hence 
\[\sum_{\begin{matrix}\sigma\in\mathfrak S_m\\\sigma(I)=J\end{matrix}}
\epsilon(\sigma)
\left(\prod_{i\in I} \frac {\partial }{\partial a_{i\sigma(i)}}\right)f \left(\prod_{j\in I^c}\frac {\partial }{\partial a_{i\sigma(i)}}\right) g
\]
\[=\epsilon(I,J) \sum_{\tau\in \mathfrak S(J)} \sum_{\tau_c\in \mathfrak S(J^c)}\epsilon(\tau) \epsilon(\tau_c)\frac{\partial^k f}{\partial a_{i_1\tau(j_1)}\dots\partial a_{i_k\tau(j_k)}} \frac{\partial^{m-k}g}{\partial a_{i'_1\tau_c(j'_1)}\dots \partial a_{i'_{m-k}\tau_c(j'_{m-k})}}
\]
\[= \epsilon(I,J) (\Delta_{I,J}(\frac{\partial}{\partial x})f\,\Delta_{I^c,J^c}(\frac{\partial }{\partial x}\big) g  \ .
\]
Formula \eqref{Deltaprod} follows by summing over $I$ and $J$. 
\end{proof}
There is a similar \emph{relative} result, allowing to compute $\Delta_{I,J} (fg)$  for $I,J$ two subsets of $\{1,2,\dots,m\}$, both of cardinality $k\leq m$. Let 
\[I=\{ i_1<i_2<\dots<i_k\},\qquad J=\{ j_1<j_2<\dots<j_k\}\ .\]
A subset $P\subset I$ (resp. $Q\subset J$) of cardinality $l\leq k$ can be uniquely written as
\[P=\{i_{p_1}<i_{p_2},\dots< i_ {p_l}\}, \quad \text{ resp. } Q=\{j_{q_1},j_{q_2},\dots, j_{q_l}\}\ .
\]
Set 
\[\epsilon(P:I, Q:J) = (-1)^{p_1+p_2+\dots +p_l} (-1)^{q_1+q_2+\dots+q_l}\ .
\]
\begin{proposition} Let $I,J$ be two subsets of $\{1,2,\dots, m\}$, both of cardinality $k\leq m$. Let $f,g$ be two smooth functions defined on $\mathbb V$. Then
\begin{equation}\label{rminorprod}
\Delta_{I,J}\left(\frac{\partial}{\partial x}\right)(fg) = \hskip-12pt\sum_{\begin{matrix} P\subset I\\ Q\subset J\\ \#P=\#Q\end{matrix}}\hskip -12pt\epsilon(P:I, Q:J)\,
\Delta_{P,Q}\left(\frac{\partial}{\partial x}\right)\! f\, \Delta_{I\smallsetminus P, J\smallsetminus Q}\left(\frac{\partial}{\partial x}\right)\! g
\end{equation}
\end{proposition}
\begin{proof} In order to calculate the left hand side of \eqref{rminorprod}, it is possible to ``freeze'' all variables $x_{ij}$ for $(i,j)\notin I\times J$. For $x\in \mathbb V$, let  \[\mathbb V_{I,J}^x =\left\{z= \begin{pmatrix}&&& \\ & &z_{ij} &\\ &&&\end{pmatrix}\in \Mat(m,\mathbb C), z_{ij} = x_{ij} \text{ for } (i,j) \notin I\times J
\right\}\ .\]
Then $\mathbb V_{I,J}^x\sim \Mat(k,\mathbb C)$. Now to compute the  left hand side of \eqref{rminorprod} at $x$, apply \eqref{Deltaprod}  to the restrictions of $f$ and $g$ to $\mathbb V^x_{IJ}$. 
\end{proof}

\begin{proposition}\label{Est}
 Let $s,t\in \mathbb C$. Then, for $f\in C^\infty(\mathbb V\times \mathbb V)$ and $x,y\in \mathbb V$, such that $x,y-x\in \mathbb V^\times$
\begin{equation}
\det(\frac{\partial}{\partial x}) \Big(\det(x)^s\det(y-x)^t f(x,y)\Big) =
\det (x)^{s-1} \det(y-x)^{t-1} \big(E_{s,t} f\big) (x,y)\,
\end{equation}
where $E_{s,t}$ is the differential operator on $\mathbb V\times \mathbb V$ given by
\[E_{s,t} f (x,y)= \sum_{k=0}^m\sum_{\begin{matrix} I,J\subset \{1,2,\dots,m\}\\\# I=\# J=k\end{matrix}}p_{I,J}(x,y;s,t)\,\Delta_{I^c,J^c}\left(\frac{\partial}{\partial x}\right)f(x,y)\]
where, for $I,J$ of cardinality $k$
\[p_{I,J}(x,y;s,t)=\]\[\sum_{0\leq l\leq k}(-1)^l
(s)_{(k-l)}(t)_l\hskip-12pt\sum_{\begin{matrix}P\subset I, Q\subset J\\\# P=\# Q=l\end{matrix}}\hskip-16pt\epsilon(P: I, Q: J) \ \Delta_{I^c\cup P,J^c\cup Q}(x)\, \Delta_{P^c,Q^c}(y-x)\ .
\]
\end{proposition}
\begin{proof} Using \eqref{Deltaprod}, the statement is equivalent to 

for any $I,J\subset \{ 1,2,\dots,n\}, \# I = \# J = k$,
\begin{equation*}
\epsilon(I,J)\,\det(x)^{-s+1} \det(y-x)^{-t+1}\,\Delta_{I,J} \left(\frac{\partial}{\partial x}\right) \Big(\det(x)^s\det(y-x)^t\Big)
\end{equation*}
 {\it a priori} defined for $x\in \mathbb V^\times, y-x\in \mathbb V^\times$ extends as  a  polynomial in $(x,y)$ equal to $p_{I,J}(x,y;s,t)$.
 
 Use  \eqref{rminorprod} to obtain 
 \[\Delta_{I,J}\left(\frac{\partial}{\partial x}\right)(\det x)^s \big(\det(y-x)\big)^t
 \]
 \[= \sum_{l=0}^k\sum_{\begin{matrix}P\subset I, Q\subset J\\ \#P=\#Q=l\end{matrix}}\hskip-12pt\epsilon(P:I,Q:J)\, \Delta_{I\smallsetminus P, J\smallsetminus Q} \left(\frac{\partial}{\partial x}\right)(\det x)^s \Delta_{P,Q} \left(\frac{\partial}{\partial x}\right)\big(\det (y-x)\big)^t
 \]
 By \eqref{Deltaminor},  \[ \det( x)^{-s+1} \Delta_{I\smallsetminus P,J\smallsetminus Q}\left(\frac{\partial}{\partial x}\right)(\det x)^s= \epsilon(I\smallsetminus P, J\smallsetminus Q)\, (s)_{k-l}\, \Delta_{I^c\cup P, J^c\cup Q}(x)\ .
 \]
 Moreover, as any constant coefficients differential operator, $\displaystyle \Delta_{K,L}(\frac{\partial}{\partial x}) $ commutes to translations, so that  again by \eqref{Deltaminor}
 
\[(\det (y-x)^{-t+1}\Delta_{P,Q}\left( \frac{\partial}{\partial x}\right) (\det (y-x))^t =
\epsilon(P,Q) (-1)^l \,(t)_l \,\Delta_{P^c,Q^c}(y-x)\ .
 \]
 Next, as $\vert I\smallsetminus P\vert + \vert P\vert = \vert I\vert$ and $\vert J\smallsetminus Q\vert + \vert Q\vert = \vert J\vert$
 \[\epsilon(P,Q)\epsilon(I\smallsetminus P, J\smallsetminus Q) = \epsilon(I,J)\ .
 \]
 It remains to gather all formul\ae  \ to  finish the proof of Proposition \ref{Est}.
\end{proof}

Let $p$ be a polynomial on $\mathbb V$, and let $q$ be the polynomial on $\mathbb V\times \mathbb V$ given by $q(x,y) = p(x-y)$. Let $f$ be a  function on $\mathbb V\times \mathbb V$. Let $g$ be the function on $\mathbb V\times \mathbb V$ defined by $g(u,v) = f(u,v-u)$ or equivalently
$g(x,x+y) = f(x,y)$. Then
\begin{equation}\label{pq}
\Big(q\big(\frac{\partial}{\partial x}, \frac{\partial}{\partial y}\big) f\Big)(x,y) = \Big(p\big(\frac{\partial}{\partial u} \big)g\Big) (x,x+y)\ .
\end{equation}
In the sequel, for commodity reason, the operator $q\big(\frac{\partial}{\partial x}, \frac{\partial}{\partial y}\big)$ will be denoted by $p\Big(\frac{\partial}{\partial x}-\frac{\partial}{\partial y}\Big)$
\begin{proposition}\label{opF} 
Let $s,t\in \mathbb C$. For any smooth function on $\mathbb V\times \mathbb V$ and for $x,y\in \mathbb V^\times$
\begin{equation}
  \det\big(\frac{\partial}{\partial x}-\frac{\partial}{\partial y}\big) \Big( (\det x)^s (\det y)^t f\Big)(x,y)=(\det x)^{s-1} (\det y)^{t-1}F_{s,t} f(x,y)
\end{equation}
where $F_{s,t}$ is the differential operator on $\mathbb V\times \mathbb V$ given by
\[F_{s,t} f (x,y)=\sum_{k=0}^m \sum_{\begin{matrix} I,J\subset \{1,2,\dots,m\}\\\# I=\# J=k\end{matrix}}  q_{I,J}(x,y;s,t)\Delta_{I^c,J^c}(\frac{\partial}{\partial x}-\frac{\partial}{\partial y})f(x,y)\]
where, for $I,J$ of cardinality $k$
\[q_{I,J}(x,y;s,t) =
\]
\[\sum_{0\leq l\leq k}(-1)^l
(s)_{(k-l)}\,(t)_l\hskip-12pt\sum_{\begin{matrix}P\subset I, Q\subset J\\ \# P=\# Q=l\end{matrix}}\hskip-16pt\epsilon(P: I, Q:J) \,\Delta_{I^c\cup P,J^c\cup Q}(x)\, \Delta_{P^c,Q^c}(y)\ .
\]
\end{proposition}

\begin{proof} Apply the change of variable formula \eqref{pq} to $p = \det$.
\end{proof}

There is a real version of these identities and they are obtained by the same method used to prove the \emph{real} Bernstein-Sato identities (see the proof of \eqref{realBS}). 

\begin{proposition}\label{Dop}
 Let $s,t\in \mathbb C$.  For any $f\in C^\infty(V\times V)$ and $x,y\in V^\times$
\begin{equation}
\begin{split}
& \left[\det \Big(\frac{\partial}{\partial x}-\frac{\partial}{\partial y}\Big)\right](\det x)^{s,\epsilon} (\det y)^{t,\eta} f(x,y)=\\ &\quad(\det x)^{s-1,-\epsilon} (\det y)^{t-1,-\eta}  F_{s,t} f(x,y)\ .
 \end{split}
\end{equation}
\end{proposition}

\section{Knapp-Stein intertwining operators}

The definition and properties of the \emph{Knapp-Stein intertwining operators} to be introduced later in this section are based on the study of the two (families of) distributions $(\det x)^{s,\epsilon}$. In a different terminology, there are the \emph{local Zeta functions} on $\Mat(n,\mathbb R)$. Many authors contributed to the study of these distributions, more generally in the context of simple Jordan algebras or in the context of prehomogeneous vector spaces (see \cite{t}, \cite{s}, \cite{g},\cite{ss},\cite{m},\cite{br},\cite{b}).  For the present situation \cite{bsz} turned out to be the most complete and most useful reference.

Let first consider  the case where $\epsilon=+1$, and write  $\vert \det x\vert^s$ instead of $(\det x)^{s,+}$. Use the notation $\mathcal S(V)$ (resp. $\mathcal S'(V)$) for the Schwartz space of smooth rapidly decreasing functions (resp. of tempered distributions) on $V$. Also define, for $s\in \mathbb C$
\begin{equation}
\Gamma_V(s) = \Gamma(\frac{s+1}{2})\dots \Gamma(\frac{s+m}{2})\ .
\end{equation}

\begin{proposition}\label{analcont} {\ }
\smallskip

i) for any  $\varphi\in \mathcal S(V)$ the integral $\int_V \varphi(x) \vert \det (x)\vert^s\, dx$ converges for $\Re (s)>-1$ and defines a tempered distribution $T_{s,+}$ on $\mathcal S( V)$.
\smallskip

ii) the $\mathcal S'(V)$-valued function $s\mapsto  T_{s,+}$ defined for $\Re(s)>-1$ can be analytically continued  as a meromorphic function on $\mathbb C$. 
\smallskip

$iii)$ the function $\displaystyle  s\longmapsto   \frac{ 1}{ \Gamma_V(s)} \,T_{s, +}$ extends as an entire function of $s$ (denoted by $\widetilde T_{s,+}$) with values in the space of tempered distributions.
\end{proposition}
\begin{proof} See \cite{bsz} and specially Theorem 5.12. A careful examination of the $\Gamma$ factors in the normalizing factor $\Gamma_V(s)$  shows that the poles are at $s=-1,-2,\dots$ if $m>1$ and  at $s=-1,-3,\dots$ if $m=1$.
\end{proof}

For $f\in \mathcal S(V)$, define the \emph{ Euclidean Fourier transform} $\mathcal Ff$ by
\[\mathcal Ff(x) = \int_V e^{-2i\pi\langle x,y\rangle} f(y)\, dy\ .
\]

 The Fourier transform is extended to various functional spaces, and in particular to the space of tempered distributions $\mathcal S'(V)$. Recall the elementary formul\ae, for $p\in \mathcal P(V)$
\begin{equation}\label{Fourier}
\mathcal F\left(p\left(\frac{\partial }{\partial x}\right) f\right)= p(2i\pi \, .\,) \mathcal F f,\qquad \mathcal F(p\,f) = p\left(-\frac{1}{2i\pi}\frac{\partial}{\partial x}\right) (\mathcal F f)
\end{equation}
\begin{proposition} The Fourier transform of the tempered distribution ${\widetilde T}_{s,+}$ is given by
\begin{equation}\label{zetaT+}
\mathcal F(\widetilde T_{s,+}) = \pi^{-\frac{m^2}{2}-ms} \, \widetilde T_{-m-s,+}
\end{equation}
or equivalently
\begin{equation}\label{zeta}
 \mathcal F \Big(\frac{1}{\Gamma_V(s)} \vert \det ( \,.\,)\vert^s \Big) =  \frac{\pi^ {-\frac{m^2}{2}-ms}}{\Gamma_V(-s-m)} \vert \det(\,.\,)\vert^{-m-s}\ .
\end{equation}
\end{proposition}
\begin{proof}
See \cite{bsz} Theorem 4.4 and Theorem 5.12.
\end{proof}

Now let $\epsilon=-1$. The corresponding results do not seem to have been written, although they could be deduced from \cite{br}. In our approach, the results for $(\det x)^{s,+}$ are used to prove those for $(\det x)^{s,-}$.

\begin{proposition}\label{analextT-} {\ }
\smallskip

$i)$ for any  $\varphi\in \mathcal S(V)$ the integral $\int_V \varphi(x) ( \det x)^{s,-}\, dx$ converges for $\Re (s)>-1$ and defines a tempered distribution $T_{s,-}$ on $\mathcal S( V)$.
\smallskip

$ii)$ the $\mathcal S'(V)$-valued function $s\mapsto  T_{s,-}$ defined for $\Re(s)>-1$ can be analytically continued  as a meromorphic function on $\mathbb C$. 
\smallskip

$iii)$ the function $ \displaystyle s\mapsto \frac{ 1}{ s\,\Gamma_V(s-1)} \,T_{s,-}$ extends as an entire function of $s$ (denoted by ${\widetilde T}_{s,-}$) with values in $\mathcal S'(V)$.
\end{proposition}

\begin{proof} As a special case of \eqref{realBS}, the following identity holds on $V^\times $
\begin{equation}\label{realBSpm}
\det\left(\frac{\partial}{\partial x}\right) (\det x)^{s+1,+} = (s+1)_m\, (\det x)^{s,-}\ .
\end{equation}
Next 
\[\frac{\Gamma_V(s+1)}{\Gamma_V(s-1)} = \frac{\Gamma(\frac{s}{2}+1)\dots \Gamma(\frac{s+m-1}{2}+1)}{\Gamma(\frac{s}{2})\dots \Gamma(\frac{s+m-1}{2})}= 2^{-m}\, (s)_m=2^{-m}\, \frac{s}{s+m}\, (s+1)_m\ .
\]
Rewrite \eqref{realBSpm} as
\[\frac{1}{s\,\Gamma_V(s-1)} (\det x)^{s,-} =2^{-m} \frac{1}{s+m} \det\left(\frac{\partial}{\partial x}\right)\Big( \frac{1}{\Gamma_V(s+1)}(\det x)^{s+1,+}\Big)\ .
\]
For $\Re s$ large enough, both sides extend as continuous functions on $V$ and hence coincide as distributions. Viewed now as a distribution-valued function of $s$, the right hand side extends holomorphically to all of $\mathbb C$ except perhaps at $s=-m$. To get the statements of Proposition \ref{analextT-}, it suffices to prove that at $s=-m$ the right hand side can be continued as a holomorphic function. In turn this is a consequence of the following lemma.
\begin{lemma}
\begin{equation}
\det \left( \frac{\partial}{\partial x}\right)\big( {\widetilde T}_{-m+1,+} \big)= 0\ .
\end{equation}
\end{lemma}
\begin{proof} The Fourier transform of the distribution $\widetilde {T}_{-m+1,+}$ is equal (up to a non vanishing constant) to $\widetilde{T}_{-1,+}$ (see \eqref{zetaT+}).  Hence the statement of the lemma is equivalent to
\begin{equation}\label{T-1}
(\det x)\, \widetilde {T}_{-1,+} = 0\ .
\end{equation}
But $\widetilde {T}_{-1,+}$ (the "first" residue of the meromorphic function  $s\mapsto T_{s,+}$) is equal (up to a non vanishing constant) to the quasi-invariant measure on the $L$-orbit $\mathcal O_1 = \{ x\in V, \rank (x) = m-1\}$ (see \cite{bsz} Theorem 5.12). As $\mathcal O_1 \subset \{x\in V, \det x =0\}$, equation \eqref{T-1} follows. 
\end{proof}
This finishes the proof of Proposition \ref{analextT-}. A careful analysis of the normalization factor $s\Gamma_V(s-1)$ shows that $T_{s,-}$ has poles at $s=-1,-2,-3,\dots$ if $m>1$, and  at $s=-2,-4,\dots$ if $m=1$.

\end{proof}

\begin{proposition}
\begin{equation}\label{zeta-}
\mathcal F(\widetilde T_{s,-}) = - i^m \pi^{-\frac{m^2}{2} -ms}\,\  \widetilde T_{-m-s,-}\ .
\end{equation}
\end{proposition}
\begin{proof} During the proof of Proposition \ref{analextT-}, it was established that
\[\widetilde T_{s,-} = 2^{-m}\frac{1}{s+m} \det\left(\frac{\partial}{\partial x}\right) T_{s+1,+}\ .
\]
Hence, using \eqref{zeta}
\[\mathcal F(\widetilde T_{s,-}) = 2^{-m}\frac{1}{s+m}\pi^{-\frac{m^2}{2}-m(s+1)}  (2i\pi)^m (\det x)\, \widetilde T_{-s-m-1,+}\ 
\]
which, for generic $s$ can be rewritten as 
\[i^m \pi^{-\frac{m^2}{2}-ms}\, \frac{1}{s+m}\, \frac{1}{\Gamma_V(-s-m-1)}\, (\det x )\, T_{-s-m-1,+}\ .
\]
Next, for $\Re(s)$ large enough, $(\det x)\, T_{s,+} = T_{s+1,-}$, and  by analytic continuation this holds for any $s$ where both sides are defined. Use this result to obtain \eqref{zeta-} for generic $s$, and  by continuity for all s.
\end{proof}

For $(s, \epsilon)\in \mathbb C\times \{\pm\}$, let
\[
\gamma(s,\epsilon) = \left\{ \begin{matrix} \frac{1}{\Gamma_V(s)}\quad &\text{if } \epsilon = 1\\ \frac{1}{s\Gamma_V(s-1)} \quad &\text{if } \epsilon = -1
\end{matrix}\right .
\]
so that
 \begin{equation}
 {\widetilde T}_{s,\epsilon} = \gamma(s,\epsilon) T_{s,\epsilon}\ .
 \end{equation}
Let
\[\rho(s, \epsilon) = \left\{\begin{matrix} \pi^{-\frac{m^2}{2}-ms}\quad  &\text{ if }\epsilon = +1\\ -i^m \pi^{-\frac{m^2}{2}-ms}\quad  &\text{ if } \epsilon = -1\end{matrix}\right .
\]
so that 
\begin{equation}\label{zetapm}
\mathcal F({\widetilde T}_{s,\epsilon}) = \rho(s,\epsilon)\, {\widetilde T}_{-s-m,\epsilon}\ .
\end{equation}

The Knapp-Stein intertwining operators play a central role in semi-simple harmonic analysis (see \cite{k} for  general results). The present approach takes advantage of the specific situation to give more explicit results.

For $(\lambda, \epsilon) \in \mathbb C\times \{\pm\}$ consider the following operator (\emph{Knapp-Stein intertwining operator}) (formally) defined by
\begin{equation}\label{KS}
J_{\lambda,\epsilon} f(x) = \int_ V  \det(x-y)^{-2m+\lambda, \epsilon}f(y) \,dy\ .
\end{equation}
The operator $J_{\lambda, \epsilon}$ verifies the following (formal) intertwining property.
\begin{proposition} For any $g\in G$,
\[J_{\lambda,\epsilon} \circ \pi_{\lambda,\epsilon}(g)= \pi_{2m-\lambda,\epsilon}(g) \circ J_{\lambda, \epsilon}\ .
\]
\end{proposition}

\begin{proof}
\[J_{\lambda,\epsilon}\big(\pi_{\lambda, \epsilon}(g)f\big) (x)= \int_V  \left( \det(x-y)\right)^{-2m+\lambda, \epsilon}\, \alpha(g^{-1},y)^{-\lambda, \epsilon} f\big(g^{-1}(y)\big)\, dy
\]
which, by using \eqref{covdet} and the cocycle property of $\alpha$ can be rewritten as
\[ \alpha(g^{-1},x)^{-2m+\lambda, \epsilon}\int_V  \det\big(g^{-1} (x) -g^{-1}(y)\big)^{-2m+\lambda,\epsilon}  \alpha( g^{-1}, y)^{-2m-\lambda+\lambda, \epsilon^2} \,dy
\] 
and use  the change of variable $z=g^{-1}(y),\  dz = \vert \alpha(g^{-1},y)\vert^{-2m}dy$ to get
\[ J_{\lambda,\epsilon}\big(\pi_{\lambda, \epsilon}(g)f\big) (x)=  \alpha(g^{-1},x)^{-(2m-\lambda),\epsilon} \int_V  \det\big(g^{-1}(x)-z\big)^{-2m+\lambda, \epsilon} f(z) \,dz 
\]
\[=\pi_{2m-\lambda, \epsilon}(g)\big(J_{\lambda,\epsilon} f\big) (x)\ .
\]
\end{proof}
To pass from a formal operator to an actual operator, notice that the Knapp-Stein operator  is a convolution operator and hence \eqref{KS} can be rewritten as
\[J_{\lambda, \epsilon} f = T_{-2m+\lambda, \epsilon} \star \ .f
\]
The study of the distributions $T_{s,\pm}$  
strongly suggests to define the \emph{normalized} intertwining operator $\widetilde {J}_{\lambda, \epsilon}$ by
\begin{equation}
\widetilde J_{\lambda, \epsilon} f= {\widetilde T}_{-2m+\lambda, \epsilon}\star f
\end{equation}
for $f\in \mathcal S(V)$, 
or more explicitly
\[\widetilde {J}_{\lambda,+} f(x) = \frac{1}{ \Gamma_V (-2m+\lambda)} \int_V \vert \det(x-y)\vert^{-2m+\lambda} f(y)\, dy\ ,
\]
\[\widetilde {J}_{\lambda, -} f(x) = \frac{1}{ (-2m+\lambda)\Gamma_V (-2m+\lambda-1)} \int_V  (\det(x-y))^{-2m+\lambda,-} f(y)\, dy\ .
\]
The representation $\pi_{\lambda, \epsilon}$ is not properly defined on $\mathcal S(V)$, but its infinitesimal version is. In fact, let $\varphi \in C^\infty_c(V)$. For $g\in G$ sufficiently close to the identity, $g$ is defined on the compact $Supp(\varphi)$, so that the following definition makes sense :  for $X\in \mathfrak g$ let \[d\pi_{\lambda, \epsilon}(X) \varphi = \left(\frac{d}{dt}\right)_{t=0} \pi_{\lambda,\epsilon} (\exp tX)\varphi\ .
\] 
Moreover, it is well known that the resulting operator $d\pi_{\lambda, \epsilon}(X)$ is a differential operator of order 1 on $V$ with polynomial coefficients, hence can be extended as a continuous operator on the Schwartz space $\mathcal S(V)$, and by duality as an operator on $\mathcal S'(V)$. An operator $J : \mathcal S(V) \rightarrow \mathcal S'(V)$ is said to be an intertwining operator w.r.t. $(\pi_{\lambda, \epsilon}, \pi_{2m-\lambda, \epsilon})$ if for any $X\in \mathfrak g$,
\[J\circ d\pi_{\lambda, \epsilon}(X) = d\pi_{2m-\lambda, \epsilon}(X) \circ J\ .
\]
The next statement is easily obtained by combining the results on the family of distributions $\widetilde T_{s,\epsilon}, (s, \epsilon) \in \mathbb C\times \{\pm\}$ (see Propositions \ref{analcont}, \ref{analextT-}), and the formal intertwining property.

\begin{proposition}\label{covJ} {\ }
\smallskip

$i)$  the operator $\widetilde J_{\lambda, \epsilon}$ is a continuous operator form $\mathcal S(V)$ into $\mathcal S'(V)$.

$ii)$ the operator $\widetilde J_{\lambda, \epsilon}$ intertwines the representations $\pi_{\lambda, \epsilon}$ and $\pi_{2m-\lambda, \epsilon}$ 

$iii)$ the (operator-valued) function $\lambda \longmapsto {\widetilde J}_{\lambda, \epsilon}$ is holomorphic.
\end{proposition}

\section{Construction of the families $D_{\lambda, \mu}$ and $B_{\lambda,\mu;k}$}
 
Recall the differential operator $F_{s,t}$ on $V\times V$, constructed in section 4. Define for $s,t\in \mathbb C$
\begin{equation}
H_{s,t} = \mathcal F^{-1}\circ F_{s,t}\circ \mathcal F
\end{equation}

As $F_{s,t}$ is a differential operator with polynomial coefficients, $H_{s,t}$ is also a differential operator with polynomial coefficients.
To be more explicit, taccording to \eqref{Fourier},  the passage from $F_{s,t}$ to $H_{s,t}$ consists in changing $p(\frac{\partial}{\partial x}, \frac{\partial}{\partial y})$ to multiplication by $p(-2i\pi x, -2 i\pi y))$, and multiplication by $p(x,y)$ to the differential operator $p\big(\frac{1}{2i\pi} \frac{\partial}{\partial x}, \frac{1}{2i\pi} \frac{\partial}{\partial y}\big)$. Observe that $q_{I,J}$ is homogeneous of degree $2m-k$ and $\Delta_{I^c,J^c}$ is homogeneous of degree $m-k$, where $k=\#I=\#J$. 
This leads to 
\begin{equation}
H_{s,t} =\left(\frac{i}{2\pi}\right)^m\sum_{k=0}^m (-1)^k \hskip -16pt\sum_{\begin{matrix} I,J\subset \{1,2,\dots,m\}\\\# I=\# J=k\end{matrix}}\hskip-24pt h_{I,J}\left(\frac{\partial}{\partial x}, \frac{\partial}{\partial y};s,t\right)\Big( \Delta_{I^c,J^c}(x-y)f(x,y)\Big)
\end{equation}
where the polynomial $h_{I,J}(\xi,\eta;s,t)$ is given by
\[
h_{I,J}(\xi,\eta;s,t)= 
\]
\[\sum_{0\leq l\leq k}
(s)_{(k-l)}\,(t)_l\hskip-15pt\sum_{\begin{matrix}P\subset I, Q\subset J\\ \# P=\# Q=l\end{matrix}}\hskip-15pt\epsilon(P: I, Q:J) \,\Delta_{I^c\cup P,J^c\cup Q}(\xi)\, \Delta_{P^c,Q^c}(\eta)\ .
\]

\begin{theorem}\label{covH}

 The operator $H_{m-\lambda, m-\mu}$ is $G$-covariant with respect to  \goodbreak $(\pi_{\lambda,\epsilon}\otimes\pi_{\mu, \eta}, \pi_{\lambda+1,-\epsilon}\otimes \pi_{\mu+1,-\eta})$. 
\end{theorem}
The (rather long) proof will be given at the end of this section. The next results are preparations for the proof.

Let $M$ be the continuous operator on $\mathcal S(V\times V)$ given by
\[M\varphi(x,y) = \det(x-y)\varphi(x,y)
\]
\begin{proposition}\label{covM}
The operator $M$ intertwines $\pi_{\lambda, \epsilon}\otimes \pi_{\mu, \eta}$ and $\pi_{\lambda-1,-\epsilon}\otimes \pi_{\mu-1,-\eta}$.
\end{proposition}

\begin{proof} Let $\varphi\in C^\infty_c(V\times V)$ .  Let  $g\in G$, and assume that $g$ is defined on $Supp(\varphi)$.  
\begin{equation*}
\begin{split}
&\Big(M\circ\big(\pi_{\lambda,\epsilon}(g)\otimes \pi_{\mu,\eta}(g)\big)\varphi\Big)(x,y) \\=\ & \det(x-y)\, \alpha(g^{-1},x)^{-\lambda, \epsilon} \,\alpha(g^{-1},y)^{-\mu,\eta}\,\varphi(g^{-1}(x),g^{-1}(y))
\end{split}
\end{equation*}
whereas
\begin{equation*}
\begin{split}
&\Big(\big(\pi_{\lambda-1,-\epsilon}(g)\otimes \pi_{\mu-1,-\eta}(g)\big) \circ M\Big) \varphi (x,y) = \\
\det\big(g^{-1}(x)&-g^{-1}(y)\big)\alpha(g^{-1},x)^{-\lambda+1,-\epsilon} \alpha(g^{-1},y)^{-\mu+1,-\eta}\varphi\big(g^{-1}(x)-g^{-1}(y)\big)\ .
\end{split}
\end{equation*}
Use \eqref{covdet} to conclude that 
 \[ \Big(M\circ\big(\pi_{\lambda,\epsilon}(g)\otimes \pi_{\mu,\eta}(g)\big)\varphi = \Big(\big(\pi_{\lambda-1,-\epsilon}(g)\otimes \pi_{\mu-1,-\eta}(g)\big) \circ M\Big) \varphi\ .
 \]
 For $X\in \mathfrak g$, and for $t$ small enough, $g_t=\exp tX$ is defined on $Supp(\varphi)$. Apply the previous result to $g_t$, differentiate w.r.t. $t$ at $t=0$ to get
 \[M\circ \big(d(\pi_{\lambda, \epsilon} \otimes \pi_{\mu, \eta})(X)\big) \varphi = \big(d(\pi_{\lambda-1, -\epsilon} \otimes \pi_{\mu-1, -\eta})(X)\big)\circ M\varphi
 \]
for any $\varphi\in C^\infty_c(V\times V)$, and extend this equality to any $\varphi$ in $\mathcal S(V\times V)$ by continuity.
\end{proof}

The next proposition is the key result towards the proof.
\begin{proposition}\label{prepprop}
 For $f\in \mathcal S(V\times V)$
\begin{equation*}\label{McircJ}
M\circ (\widetilde J_{\lambda,\epsilon}\otimes \widetilde J_{\mu,\eta}) f\ = d\big((\lambda,\epsilon),(\mu,\eta)\big)
\,\big((\widetilde J_{\lambda+1, -\epsilon} \otimes \widetilde J_{\mu+1, -\eta}) \circ H_{-m+2\lambda, -m+2\mu}\big) f,
\end{equation*}
where $d\big((\lambda,\epsilon),(\mu,\eta)\big)$ is equal to
 \begin{align*}
 &\frac{\pi^{4m^2}}{(\lambda-m)\dots(\lambda-2m+2)(\mu-m)\dots(\mu-2m+2)}& \qquad \epsilon =+1,\eta = +1\\
&\frac{2^{-m} \pi^{4m^2}}{(\lambda-m)\dots(\lambda-2m+2)(\mu-m)}
&\qquad \epsilon =+1,\eta = -1\\
&\frac{2^{-m} \pi^{4m^2}}{(\lambda-m)(\mu-m)\dots(\mu-2m+2)}
&\qquad \epsilon =-1,\eta = +1\\
&\frac{2^{-2m} \pi^{4m^2}}{(\lambda-m)(\mu-m)}&\qquad \epsilon = -1,\eta=-1\ .
 \end{align*}
\end{proposition}

\begin{proof}
As the operators $\widetilde J_{\lambda, \epsilon}$ and $\widetilde J_{\mu, \eta}$  are convolution operators by a tempered distribution, the left hand side is well defined as a tempered distribution on $V\times V$, and so is its Fourier transform.

In order to alleviate the proof,  $c_1,\dots, c_4$ are used during the proof to mean complex numbers depending on $\lambda, \epsilon, \mu, \eta$ but neither on $f$ nor  on $(x,y)\in V\times V$. Their actual values are listed at the end of the computation.
By \eqref{zeta},
\begin{equation}\label{FourierMJJ}
\mathcal F\big((\widetilde J_{\lambda,\epsilon}\otimes \widetilde J_{\mu,\eta}) f\big)(x, y)= \mathcal F({\widetilde T}_{-2m+\lambda, \epsilon})(x)  \mathcal F({\widetilde T}_{-2m+\mu, \eta})(y) \mathcal Ff(x,y)
\end{equation}
\[= c_1{\widetilde T}_{ m-\lambda, \epsilon}(x) {\widetilde T}_{ m-\mu, \eta}(x)\mathcal Ff(x,y)\ .
\]
Next, for $p$ a polynomial on $V\times V$, and $\Phi\in \mathcal S'(V)$,
\[\mathcal F (p\Phi) (x,y) = p\big((-2i\pi)^{-1}\frac{\partial}{\partial x}, (-2i\pi)^{-1}\frac{\partial}{\partial y}\big) (\mathcal F\Phi)(x,y)\ .
\]
Hence 
\begin{equation}
\begin{split}
&\mathcal F\big(M\circ (\widetilde J_{\lambda,\epsilon}\otimes \widetilde J_{\mu,\eta}) f\big) (x,y)  \\  = &c_1c_2\det\left(\frac{\partial}{\partial x}-\frac{\partial}{\partial y}\right) \left( (\det x)^{m-\lambda, \epsilon}(\det y)^{m-\mu, \eta} \mathcal Ff(x,y)\right)\ .
\end{split}
\end{equation}

Assume temporarily that $ \Re \lambda, \Re \mu <<0$ so that $(\det x)^{m-\lambda, \epsilon}( \det y)^{m-\mu, \eta}$ is a sufficiently many times differentiable function on $V\times V$. Then, use  Proposition \ref{Dop} to get
\begin{equation}
\begin{split}
&\mathcal F\big(M\circ (\widetilde J_{\lambda, \epsilon}\otimes \widetilde J_{\mu, \eta}) f\big) (x,y) \\ &c_1c_2(\det x)^{m-(\lambda+1), -\epsilon)}(\det y)^{m-(\mu+1), -\eta} F_{m-\lambda, m-\mu} \left(\mathcal Ff\right)(x,y)\ ,
\end{split}
\end{equation}
the equality being valid \emph{a priori} on $V^\times \times V^\times$, but thanks to the assumption on $\lambda$ and $\mu$ it extends to all of $V\times V$. Next, by the definition of the operator $H_{s,t}$,
\begin{equation}
\begin{split}
&\mathcal F(M\circ ({\widetilde J}_{\lambda, \epsilon}\otimes {\widetilde J}_{\mu, \eta})f) (x,y)  \\
  = & c_1c_2(\det x)^{m-\lambda-1, -\epsilon)}(\det y)^{m-\mu-1, -\eta} \mathcal F \big(H_{-m+2\lambda, -m+2\mu} f\big)(x,y)\\
&c_1c_2c_3\,{\widetilde T}_{m-\lambda-1, -\epsilon}(x)\, {\widetilde T}_{m-\mu-1, -\eta}(y)\, \mathcal F \big(H_{m-\lambda, m-\mu} f\big)(x,y)\ .
\end{split}
\end{equation}
Use inverse Fourier transform and \eqref{zetapm} to conclude that 

\begin{equation}
M\circ (\widetilde J_\lambda\otimes \widetilde J_\mu) f\ = c_1c_2c_3c_4\,\left((\widetilde J_{\lambda+1, -\epsilon} \otimes \widetilde J_{\mu+1, -\eta}) \circ H_{m-\lambda, m-\mu}\right) f\ .
\end{equation}
The values of the constants $c_1,2,c_3$ and $c_4$ are given by
\begin{equation*}\begin{split}
c_1 &=\rho(-2m+\lambda,\epsilon)\, \rho(-2m+\mu,\eta)\\
c_2&= (-1)^m (2\pi)^{-2m} \gamma(m-\lambda, \epsilon)\, \gamma(m-\mu ,\eta)\\
c_3&= \frac{1}{\gamma(m-\lambda-1, -\epsilon)\,\gamma(m-\mu-1,  -\eta)} \\
c_4 &= \frac{1}{\gamma(\lambda+1,-\epsilon)\,\gamma(\mu+1,-\eta)}
\end{split}
\end{equation*}
 so that $c_1c_2c_3c_4$ is equal to
 \begin{equation*}
 \begin{split}
 \frac{\pi^{4m^2}}{(\lambda-m)\dots(\lambda-2m+2)(\mu-m)\dots(\mu-2m+2)}& \qquad \epsilon =+1,\eta = +1\\
\frac{2^{-m} \pi^{4m^2}}{(\lambda-m)\dots(\lambda-2m+2)(\mu-m)}
&\qquad \epsilon =+1,\eta = -1\\
\frac{2^{-m} \pi^{4m^2}}{(\lambda-m)(\mu-m)\dots(\mu-2m+2)}
&\qquad \epsilon =-1,\eta = +1\\
\frac{2^{-2m} \pi^{4m^2}}{(\lambda-m)(\mu-m)}&\qquad \epsilon = -1,\eta=-1
 \end{split}
 \end{equation*}
By analytic continuation, \eqref{McircJ} holds for all $\lambda, \mu$, thus proving Proposition \ref{prepprop}. Incidentally, notice that the last step implies the vanishing of \break $\big((\widetilde J_{\lambda+1, -\epsilon} \otimes \widetilde J_{\mu+1, -\eta}) \circ H_{-m+2\lambda, -m+2\mu}\big)$ at the poles of $d\big((\lambda, \epsilon), (\mu, \eta)\big)$. 

\end{proof}

To finish the proof of Theorem \ref{covH}, note that, by Lemma \ref{covM} and Proposition \ref{covJ} the operator $M\circ (\widetilde J_{\lambda, \epsilon}\otimes \widetilde J_{\mu,\eta})$ is covariant with respect to\break $(\pi_{\lambda, \epsilon}\otimes \pi_{\mu, \eta}), (\pi_{2m-\lambda-1, -\epsilon} \otimes \pi_{2m-\mu-1,-\eta})$. Using Proposition \ref{prepprop}, this implies, generically in $(\lambda, \mu)$ that for any $f\in C^\infty_c(V\times V)$
 and any $g\in G$ which is defined on $Supp(f)$,
 \begin{equation*}
\begin{split}
\big((\widetilde J_{\lambda+1,-\epsilon}\otimes \widetilde J_{\mu+1,-\eta})&\circ (\pi_{\lambda+1,-\epsilon}(g)\otimes \pi_{\mu+1,-\eta}(g))\circ H_{-m+2\lambda, -m+2\mu} \big)f\\
 = \big((\widetilde J_{\lambda+1,-\epsilon}\otimes \widetilde J_{\mu+1,-\epsilon})&\circ H_{m-\lambda, m-\mu}\circ ( \pi_{\lambda, \epsilon}(g)\otimes \pi_{\mu,\eta}(g)\big) f
\end{split}
\end{equation*}
Generically in $(\lambda, \mu)$, the convolution operator $\widetilde J_{\lambda+1,-\epsilon}\otimes \widetilde J_{\mu+1,-\eta}$ is injective on $C^\infty_c(V)$ as can be seen after performing a Fourier transform, so that 
\[\Big(\big(\pi_{\lambda+1,-\epsilon}(g)\otimes \pi_{\mu+1,-\eta}(g)\big)\circ H_{m-\lambda, m-\mu} \Big)f\]\[=\Big(H_{m-\lambda, m-\mu}\circ \big( \pi_{\lambda, \epsilon}(g)\otimes \pi_{\mu,\eta}(g)\big)\Big)f\ .
\]
The covariance of $H_{m-\lambda, m-\mu}$ follows, at least generically in $\lambda, \mu$ and hence everywhere by analytic continuation. This completes the proof of Theorem \ref{covH}.

For convenience in the sequel, let shift the parameters in the notation by setting 
\[D_{\lambda, \mu} = H_{m-\lambda, m-\mu}\ .
\]
Perhaps is it enlightening to  state a version of  Theorem \ref{covH} in the compact picture. Going back to the notation of the Introduction, the tensor product $\mathcal E_{\lambda, \epsilon}\boxtimes \mathcal E_{\mu, \eta}$ can be completed to  a space $\mathcal E_{(\lambda, \epsilon), (\mu,\eta)}$ of smooth sections of the line bundle $E_{\lambda, \mu} \boxtimes E_{\mu,\eta}$ over $X\times X$. The operator $M$ can also be transferred as a continuous operator from $\mathcal E_{(\lambda, \epsilon), (\mu,\eta)}$ into $\mathcal E_{(\lambda-1, -\epsilon), (\mu-1,-\eta)}$.
Denote by $\widetilde I_{\lambda, \epsilon} : \mathcal E_{\lambda, \epsilon}$ into $\mathcal E_{2m-\lambda, \epsilon}$ the normalized Knapp-Stein operator, which corresponds to $\widetilde J_{\lambda, \epsilon}$ in the principal chart. The formulation to be given below is a consequence of Theorem \ref{covH}, using the well-known fact that the Knapp-Stein intertwining operators are  invertible, at least generically in $\lambda$, the inverse of $\widetilde I_{\lambda, \epsilon}$ being equal (up to a scalar) to $\widetilde I_{ 2m-\lambda, \epsilon}$.

\begin{theorem}\label{covHc}
 The operator $D_{(\lambda, \epsilon), (\mu,\eta)}$ defined as
\[D_{(\lambda, \epsilon), (\mu,\eta)} = \left(\widetilde I_{2m-\lambda-1,-\epsilon}\otimes \widetilde I_{2m-\mu-1,-\eta}\right)\circ M\circ \left(\widetilde I_{\lambda,\epsilon} \otimes \widetilde I_{\mu,\eta}\right)
\]
which, by construction intertwines $\pi_{\lambda, \epsilon}\otimes \pi_{\mu, \eta}$ and $\pi_{\lambda+1,-\epsilon}\otimes \pi_{\mu+1, -\eta}$ (as representations of $G$) is a \emph{ differential operator} on $X\times X$.
\end{theorem}

Let $\res: C^\infty(V\times V) \longrightarrow C^\infty(V)$ be the \emph{restriction map} defined by
\[ \res (\varphi) (x) =  \varphi(x,x)\ .
\] 
For any $\lambda,\epsilon$ and $\mu,\eta$ in $\mathbb C\times \{\pm\}$, the restriction map intertwines the representations $\pi_{\lambda,\epsilon} \otimes \pi_{\mu, \eta}$ and $\pi_{\lambda+\mu, \epsilon\eta}$.

Let $\lambda, \mu \in \mathbb C$, and $k\in \mathbb N$. Let $B_{\lambda,\mu, k} : C^\infty(V\times V) \longrightarrow C^\infty(V)$ be the bi-differential operator defined by
\[B_{\lambda, \mu; k}= \res \circ D_{\lambda+k-1, \mu+k-1} \circ \dots \circ D_{\lambda, \mu}\ .
\]
The covariance property of the operators $D_{\lambda, \mu}$ and of $\res$ imply the following result.

\begin{theorem}  Let $(\lambda, \epsilon), (\mu, \eta)$ be in $\mathbb C\times \{\pm\}$. The operator $B_{\lambda, \mu; k}$ is covariant w.r.t. $(\pi_{\lambda,\epsilon} \otimes \pi_{\mu, \eta}, \pi_{\lambda+\mu+2k, \epsilon \eta})$.
\end{theorem}

A remarkable fact is that whereas the operator $H_{\lambda, \mu}$ has polynomial functions as coefficients, the operator $B_{\lambda,\mu; k}$ has \emph{constant} coefficients, i.e. is of the form
\[\varphi \longmapsto \sum_{\boldsymbol \alpha, \boldsymbol \beta} a_{\boldsymbol \alpha, \boldsymbol \beta}\, \left(\frac{{\partial}^{\vert \boldsymbol \alpha\vert+\vert \boldsymbol \beta\vert}}{{\partial y}^{\boldsymbol \alpha}{\partial z}^{\boldsymbol \beta}} \,\varphi\right) (x,x)\]
where $a_{\boldsymbol \alpha, \boldsymbol \beta}$ are complex numbers.
In fact, this is merely a consequence of the invariance of the $B_{\lambda,\mu; k}$ under the action of the translations (action of $\overline{N}$). More concretely,this is due to the vanishing on the diagonal $diag(V)$ of many of the coefficients of the operators $H_{\lambda, \mu}$. It seems however difficult to find a closed formula for the coefficients of $B_{\lambda, \mu;k}$ except if $m=1$ .

\section{The case $m=1$ and the $\Omega$-process}
For $m=1$, a simple calculation yields
\begin{equation}
F_{s,t}f = (-tx+sy)f + xy\left( \frac{\partial^2}{\partial x\partial y}\right)f
\end{equation}
\begin{equation}
H_{s,t} f = \frac{1}{2i\pi}\left(\big(-(t-1)\frac{\partial}{\partial x} f+ (s-1)\frac{\partial} {\partial y}f\big)-(x-y) \frac{\partial^2f}{\partial x\partial y})\right)\ .
\end{equation}

\begin{equation}
D_{\lambda, \mu} = \frac{1}{2i\pi}\left( \mu\frac{\partial}{\partial x}-\lambda\frac{\partial}{\partial y} -(x-y) \frac{\partial^2}{\partial x\partial y}\right)
\end{equation}

There is a relation with the \emph{$\Omega$-process}, which we now recall following the classical spirit (see e.g.\cite{o}), but in terms adapted to our situation.

Let $(\lambda, \epsilon) \in \mathbb C\times \{\pm\}$ and let $\mathcal F_{\lambda, \epsilon}$ be the space of smooth functions defined on $\mathbb R^2 \smallsetminus \{0\}$  which satisfy
\[\forall t\in \mathbb R^*\qquad F(tx_1,tx_2) = t^{-\lambda, \epsilon} F(x_1,x_2)\ .
\]

To $F\in \mathcal F_{\lambda, \epsilon}$ associate the function $f$ given by $f(x) = F(x,1)$. Then $f$ is a smooth function on $\mathbb R$, and $F$ can be recovered from $f$ by
\[F(x_1,x_2) = x_2^{-\lambda, \epsilon} f(\frac{x_1}{x_2})\ ,
\]
at least for $x_2\neq 0$ and then extended by continuity.

Let $g\in SL_2(\mathbb R)$ and let $g^{-1}= \begin{pmatrix}a&b\\ c&d\end{pmatrix}$. The function $F\circ g^{-1}$ also belongs to $\mathcal F_{\lambda, \epsilon}$, and is explicitly given by 
\[F\circ g^{-1} (x_1,x_2) = F(ax_1+bx_2,cx_1+dx_2)
\]
Its associated function on $\mathbb R$ is given by
\[(F\circ g^{-1})(x,1) = F(ax+b,cx+d) = (cx+d)^{-\lambda, \epsilon} f\left(\frac{ax+b}{cx+d}\right), 
\]
so that the natural action of $G=SL(2,\mathbb R)$  on $\mathcal F_{\lambda,\epsilon}$ is but another realization of the representation $\pi_{\lambda, \epsilon}$.

Now let $(\lambda, \epsilon), (\mu, \eta)\in \mathbb C\times \{\pm\}$ and consider the space $\mathcal F_{(\lambda, \epsilon), (\mu,\eta)}$ of smooth functions $F$ on $\mathbb R^2\smallsetminus \{0\} \times \mathbb R^2\smallsetminus \{0\} $ which satisfy 
\[\forall t,s\in \mathbb R^*,\qquad F(t(x_1,x_2),s(y_1,y_2)) = t^{-\lambda, \epsilon}s^{-\mu,\eta}F\big((x_1,x_2),(y_1,y_2)\big)\ .
\]
The group $SL_2(\mathbb R)$ acts naturally (diagonally) on $\mathcal F_{(\lambda, \epsilon), (\mu,\eta)}$, and this action yields a realization of $\pi_{\lambda, \epsilon}\otimes \pi_{\mu, \eta}$. More explicitly,
let $f(x,y) = F((x,1),(y,1)$. Then for $g\in SL_2(\mathbb R)$ such that 
$g^{-1} = \begin{pmatrix} a&b\\c&d\end{pmatrix}$
\[F\circ g^{-1}((x,1),(y,1)) = (cx+d)^{-\lambda, \epsilon} (cy+d)^{-\mu,\eta} f\left (\frac{ax+b}{cx+d}\,, \frac{ay+b}{cy+d}\right)\ .
\]

The polynomial $\det\begin{pmatrix}x_1&y_1\\x_2&y_2 \end{pmatrix}$ is invariant by the action of $SL_2(\mathbb R)$ and so is the differential operator
\begin{equation*}
\Omega = \frac{\partial^2}{\partial x_1\partial y_2}-\frac{\partial^2}{\partial x_2\partial y_1}\ .
\end{equation*}
The operator $\Omega$ maps $\mathcal F_{(\lambda, \epsilon),(\mu, \eta)}$ to $\mathcal F_{(\lambda+1, -\epsilon),(\mu+1,-\eta)}$ and yields a covariant differential w.r.t. $(\pi_{\lambda, \epsilon} \otimes \pi_{\mu,\eta}, \pi_{\lambda+1, -\epsilon} \otimes \pi_{\mu+1,-\eta})$. 

Let $F\in \mathcal F_{(\lambda,\epsilon),(\mu,\eta)}$. As above, let $f$ be the function on $\mathbb R\times \mathbb R$ obtained by deshomogenization of $F$ i.e. $f(x,y) = F\big((x,1),(y,1)\big)$. The corresponding differential operator on $\mathbb R\times \mathbb R$ is given by
\[
\omega_{\lambda, \mu} f(x,y)) = \big(\Omega F\big)\big((x,1),(y,1)\big) =-\mu\frac{\partial f}{\partial x} +\lambda \frac{\partial f}{\partial y} +(x-y) \frac{\partial^2 f}{\partial x\partial y}\ ,
\]
independently of $\epsilon$ and  $\eta$,
so that
$D_{\lambda, \mu} = -2i\pi \omega_{\lambda, \mu}
$.

For $k\in \mathbb N$, let $R_k: C^\infty(\mathbb R^2\times \mathbb R^2)  \longmapsto C^\infty(\mathbb R^2)$ be the bi-differential operator  given by $R_k = \res \circ \,\Omega^k$ or more explicitely
\begin{equation}
x\in V,\qquad R_kF(x) = \Omega^kF(x,x)
\end{equation}
 The operator $R_k$ commutes to the action of $SL(2,\mathbb R)$. If $F$ belongs to $\mathcal F_{(\lambda, \epsilon), (\mu,\eta)}$, the function $R_k F$ is homogeneous of degree $(\lambda+\mu+2k,\epsilon\eta)$. By deshomogenization, the corresponding operator is
\[r_{\lambda, \mu; k} = \res\circ \,\omega_{\lambda+k-1,\mu+k-1} \circ \dots \circ \omega_{\lambda, \mu}
\]
so that $B_{\lambda, \mu;k} = (-2i\pi)^k r_{\lambda, \mu:k}$.

A classical computation in the theory of the $\Omega$-process yields an explicit expression for $r_{\lambda, \mu, k}$
\begin{equation}\label{RC}
r_{\lambda, \mu;k} =\res\circ \left(k!  \sum_{i+j=k} (-1)^j \begin{pmatrix} -\lambda-i\\j\end{pmatrix}\begin{pmatrix} -\mu-j\\i\end{pmatrix}\frac{\partial^k}{\partial x^i\partial y^j}\right)\ .
\end{equation}
The computation can be found in \cite{ops}, where the indices $\lambda$ and $\mu$ are supposed to be negative integers, but the computation goes through without this assumption.

Two special cases are worth being reported, both corresponding to cases where the representations $\pi_{\lambda,\epsilon}, \pi_{\mu,\eta}$ are \emph{reducible}.

Suppose that $\lambda = k\in \mathbb Z$. Choose $\epsilon=(-1)^k$, so that for any $t\in \mathbb R^*, t^{\lambda, \epsilon} = t^k$. Then for $g\in G$ such that $g^{-1} = \begin{pmatrix}a&b\\c&d\end{pmatrix}$
\[\pi_{k,(-1)^k}(g) f(x) = (cx+d)^{-k}f\left(\frac{ax+b}{cx+d}\right)\ .
\]
Let first consider the case where $\lambda\in -\mathbb N$, say $\lambda = -l, l\in \mathbb N$. Then the space $\mathcal P_l$ of polynomials of degree less than $l$ is preserved by the representation $\pi_{-l,(-1)^l}$
Similarly, let $\mu=-m$ for some $m\in \mathbb N$. Let $p\in \mathcal P_l, q\in \mathcal P_m$. Let $P$ (resp. $Q$) be the homogeneous polynomial on $\mathbb R^2$  obtained by homogenization of $p$ (resp.$q$). For $k\leq \inf(l,m)$, the function $R_k(P\otimes Q)$ is a polynomial which is homogeneous of degree $l+m-2k$  and which in the classical theory of invariants is called the \emph{$k^{th}$ transvectant} of $P$ and $Q$ usually denoted by $[P,Q]_k$. So $B_{-l,-m;k}$ just expresses the $k$-th transvectant at the level of inhomogeneous polynomials.

Now suppose that $\lambda = l, l\in \mathbb N$. Then restrictions  of holomorphic functions to $\mathbb R$ are preserved by the representation  $\pi_{l, (-1)^l}$. Suppose also $\mu=m\in \mathbb N$. Then the operators $D_{l,m}$ and $B_{l,m,k}$, extended as holomorphic differential operators are still covariant under the action of $G$. If $f$ is an automorphic form of degree $l$ and $g$ of degree $m$, then the covariance property of $B_{l,m;k}$ implies that  $B_{l,m,k}(f\otimes g)$  is an automorphic form of degree $l+m+2k$. The operators $B_{l,m;k}$ essentially coincide with the \emph{Rankin-Cohen brackets}, as easily deduced from formula \eqref{RC}. 

\section{The general case and some open problems}

When $m\geq 2$, the $\Omega$-process can be extended along the same lines (see \cite{ops}). Let  $\mathcal F_{\lambda,\epsilon}$  be the space of functions $F : V\times V$ which are \emph{determinantially homogeneous of weight $(\lambda, \epsilon)$}, i.e. satisfying
\[\forall \gamma\in GL(V)\qquad  F(x\gamma, y\gamma) = (\det \gamma)^{-\lambda, \epsilon} F(x,y)\ .\]
To such a function $F$, associate the function $f$ on $V$ defined by $f(x) = F(x, \bf {1}_m)$. Then $F$ can be recovered from $f$ by
\begin{equation}\label{Ff}
F(x,y) = (\det y)^{-\lambda, \epsilon} f(xy^{-1})\ ,
\end{equation}
at least when $y\in V^\times$ and everywhere by continuity. 

The group $G=SL(2m,\mathbb R)$ acts on $V\times V$ by left multiplication, i.e. if $g=\begin{pmatrix} a&b\\c&d\end{pmatrix}$
 \[
 \big(g,(x,y) \big) \longmapsto  g\begin{pmatrix} x\\y\end{pmatrix}= \begin{pmatrix} ax+by\\ cx+dy\end{pmatrix} \ .
 \]
 The determinantial homogeneity of functions is preserved by this action, and hence the representation of $G$ on $\mathcal F_{\lambda, \epsilon}$ is but another realization of $\pi_{\lambda, \epsilon}$ as can be seen by transferring the action through the correspondance $F\mapsto f$ given by \eqref{Ff}. Using this time the polynomial $\det_{2m}\begin{pmatrix} x_1&y_1\\x_2&y_2\end{pmatrix}$, an operator $\Omega$ can be defined along the same line as in the case $m=1$.  As the action of $G$ commutes to the action (on the right) of $GL(V)$, $\Omega^k$ maps $\mathcal F_{\lambda, \epsilon} \otimes F_{\mu,\eta}$ into $\mathcal F_{\lambda+1,-\epsilon} \otimes F_{\mu+1,-\eta}$ and is covariant for the action of $G$. Again, using the correspondence $F\mapsto f$, $\Omega$ lifts to a  differential operator on $V\times V$ which is covariant w.r.t. $(\pi_{\lambda,\epsilon}\otimes \pi_{\mu,\eta}, \pi_{\lambda+1, -\epsilon}\otimes \pi_{\mu+1,-\eta})$ and which can be used for defining the covariant bi-differential operators.
 
In case $m\geq 2$, it is not clear wether the two approaches coincide, as computations get very complicated. One way of attacking this question would be to determine the dimension of the space of covariant bi-differential operators w.r.t. $(\pi_{\lambda, \epsilon}\otimes \pi_{\mu, \eta}, \pi_{\lambda + \mu+2k, \epsilon, \eta})$, and/or the dimension of the space of covariant differential operators on $V\times V$ w.r.t. $(\pi_{\lambda, \epsilon}\otimes \pi_{\mu, \eta}, \pi_{\lambda+1, -\epsilon}\otimes \pi_{\mu+1, -\eta})$. If the answer (to either question) were one (at least generically in $(\lambda, \mu)$), then the two approaches would coincide.

\bigskip

\footnotesize{ \noindent Address\\  Institut \'Elie Cartan, Universit\'e de Lorraine 54506 Vand\oe uvre-l\`es Nancy (France)
\smallskip

\noindent \texttt{{jean-louis.clerc@univ-lorraine.fr 
}}

\end{document}